\documentclass[a4paper,11pt,openany,reqno]{amsart}
	\usepackage[latin1]{inputenc}
    \usepackage[T1]{fontenc}
    \usepackage[english]{babel}
    \usepackage{appendix}
    \usepackage[english]{minitoc}
    \usepackage[nice]{nicefrac}
    \usepackage[PetersLenny]{fncychap}
\usepackage[top=2cm, bottom=2cm, left=2cm, right=2cm]{geometry}
\numberwithin{equation}{section}  \makeatletter\@addtoreset{equation}{section}
	\usepackage{amsmath}
    \usepackage{amssymb}
    \usepackage{amsbsy} 
    \usepackage{latexsym, amsfonts}
    \usepackage{graphics}
    \usepackage{ulem}
    \usepackage{hhline}
    \usepackage{dsfont}
    \usepackage{mathrsfs}
    \usepackage{color}
    \usepackage{fancyhdr}
    \usepackage{rotating}
    \usepackage{fancybox}
    \usepackage{colortbl}
    \usepackage{pifont}
    \usepackage{setspace}
    \usepackage{enumerate}
    \usepackage{multicol}
    \usepackage{varioref}
    \usepackage{lmodern}
    \usepackage{textcomp}
    \usepackage{euscript}
    \usepackage[pdftex]{hyperref}
			\newtheorem{theorem}{Theorem}[section]

			\newtheorem{corollary}[theorem]{Corollary}
			\newtheorem{remark}[theorem]{Remark}

\newcommand{\norm}[1]{\left\Vert#1\right\Vert}  \newcommand{\scal}[1]{\left<#1\right>}
   
\newcommand{\R}{\mathbb{R}}    \newcommand{\C}{\mathbb{C}}

     \newcommand{\bz}{\overline{z}} \newcommand{\bw}{\overline{w}}
      
    \newcommand{\pbz}{\partial_{\bz}} \newcommand{\magn}{\nu} \newcommand{\gauss}{\mu}
\begin{document}
\title{Mehler's formulas for the univariate complex Hermite polynomials and applications}
\thanks{This research work was partially supported by a grant from the Simons Foundation and by the Hassan II Academy of Sciences and Technology.}

\author[A. Ghanmi]{Allal Ghanmi}
\dedicatory{Dedicated to Professor Ahmed Intissar on the occasion of his 65th birthday }
\address{Center of Mathematical Research and Applications of Rabat (CeReMAR),\newline
         Analysis and Spectral Geometry (A.G.S.),
          Laboratory of Mathematical Analysis and Applications (L.A.M.A.),\newline
          Department of Mathematics, P.O. Box 1014,  Faculty of Sciences, 
          Mohammed V University in Rabat, Morocco}
\email{ag@fsr.ac.ma}

\begin{abstract}
We give two widest Mehler's formulas for the univariate complex Hermite polynomials $H_{m,n}^\magn $, by performing double summations involving the products $u^m H_{m,n}^\magn  (z,\bz ) \overline{H_{m,n}^\magn  (w,\bw )}$ and  $u^m v^n H_{m,n}^\magn  (z,\bz ) \overline{H_{m,n}^{\magn'}  (w,\bw )}$. They can be seen as the complex analogues of the classical Mehler's formula for the real Hermite polynomials.
The proof of the first one is based on a generating function giving rise to the reproducing kernel of the generalized Bargmann space of level $m$.
The second Mehler's formula generalizes the one appearing as a particular case of the so-called Kibble-Slepian formula.
The proofs, we present here are direct and more simpler. Moreover, direct applications are given and remarkable identities are derived.
\end{abstract}
\maketitle

\section{Introduction } \label{s1}

The classical Mehler's formula \cite{Mehler1866,Rainville71,Andrews},
\begin{align}\label{MehlerkernelHn}
  \sum\limits_{n=0}^{+\infty} \frac{ t^n H_{n}(x) H_{n}(y)}{2^n n!}
= \frac{1}{\sqrt{1 - t^2}}  \exp\left( \frac{- t^2 (x^2 + y^2) + 2 t x y  }{1 - t^2}  \right) =: E_t(x,y) ,
\end{align}
for the univariate real Hermite polynomials
$
H_{n}(x) : = (-1)^n e^{x^2} \frac{d^n}{dx^n}(e^{-x^2}),
$
as well as its numerous versions are useful and have wide applications in many fields of mathematics and theoretical physics. 

In the present paper, we deal with two kinds of generalizations of this formula to the class of the weighted univariate complex Hermite polynomials
          \begin{align}\label{gchpmu}
         H_{m,n}^\magn  (z,\bz )=(-1)^{m+n}e^{\magn  z\bz }\dfrac{\partial ^{m+n}}{\partial \bz^{m} \partial z^{n}} \left(e^{-\magn  z \bz }\right) ; \quad z\in \C,
         \end{align}
for an arbitrary fixed positive real number $\magn >0$.
The main results to which is aimed this paper are Theorems \ref{thm:LikeMehler} and \ref{MehlerKernel}.
 Both obtained formulas are realized as double summation over the indices.
The first one corresponds to the product $u^m H_{m,n}^\magn  (z,\bz ) \overline{H_{m,n}^\magn  (w,\bw )}$.
Its proof lies on a generating function giving rise to the reproducing kernel of the generalized Bargmann space of level $m$, defined as $L^2$-eigenspace of a specific magnetic Laplacian (see Section 3).
The second obtained Mehler's formula involves the product $u^m v^n H_{m,n}^\magn  (z,\bz ) \overline{H_{m,n}^{\magn'}  (w,\bw )}$, with different arguments $\magn$ and $\magn'$, and generalizes ($\magn=\magn'=1$) the Mehler's formula given firstly in \cite{Wunsche1999} (without proof) and proved in \cite[Theorem 3.3]{IsmailTrans2016}. In fact, it appears in \cite{IsmailTrans2016} as a particular case of the so-called Kibble-Slepian formula \cite[Theorem 1.1]{IsmailTrans2016}. The proof, we present here is more direct and simpler.
We also provide application for each obtained Mehler's formula. The first one (Theorem \ref{thm:Heatkernel}) concerns a closed expression of the Heat kernel function associated to a special magnetic Laplacian on the complex plane. The second application (Theorem \ref{thm:MehlerCons}) is an integral reproducing property of the univariate complex Hermite polynomials $H_{m,n}^{\magn}$ (self-reciprocity property) by a like Fourier transform.
Curious and remarkable identities involving the univariate complex Hermite polynomials are also derived.

\section{Preliminaries on the univariate complex Hermite polynomials} \label{s2}

Notice first that for $\magn=1$, the polynomials in \eqref{gchpmu} are those
introduced by It\^o in the context of complex Markov process \cite{Ito52}, and next used as a basic tool in the nonlinear analysis of travelling wave tube amplifiers. They appear in calculating the effects of nonlinearities in broadband radio frequency communication systems \cite{Barrett}.
 Such class have been the subject of a considerable number of papers in the recent years (see for examples \cite{Gh13ITSF,BenahmadiGElkachkouri2017,
 IsmailTrans2016} and the references therein). The incorporation of the parameter $\magn $ in \eqref{gchpmu} is fairly interesting for its physical meaning. In fact, it can be interpreted as the magnitude of a constant magnetic field applied perpendicularly on the Euclidean plane.
 Below, we recall some needed results of these polynomials. We begin with the well-established fact (\cite{BenahmadiGElkachkouri2017})
  \begin{align}\label{intGaussab}
  \int_{\C} e^{-\gauss|\xi|^2 + \alpha \xi + \beta \overline{\xi} } d\lambda(\xi) =  \left(\frac{\pi}{\gauss}\right) e^{\frac{\alpha\beta}{\gauss}},
  \end{align}
  valid for every fixed positive real number $\gauss>0$ and arbitrary complex numbers $\alpha,\beta\in\C$. Here $d\lambda(\xi)=dxdy$; $\xi=x+iy\in\C$, denotes the classical Lebesgue measure on the complex plane $\C$. Formula \eqref{intGaussab} is quite easy to check by writing $\xi$ as $\xi =x+iy$; $x,y\in\R$, and next making use of the Fubini's theorem as well as the explicit formula for the gaussian integral 
  \begin{align}\label{GaussIntegral}
  \int_{\mathbb{R}} e^{-\gauss x^2 + bx} dx = \left(\frac{\pi}{\gauss}\right)^{\frac{1}{2}}e^{\frac{b^2}{4\gauss}}; \quad \gauss>0, \, b\in\C.
  \end{align}
  Based on \eqref{intGaussab}, we can reintroduce the class of univariate complex Hermite polynomials $H_{m,n}^{\magn }(z;\bz)$ by
 considering the integral representation (\cite{BenahmadiGElkachkouri2017})
 \begin{align}\label{intHermite}
  H_{m,n}^{\magn}(z;\bz) = \left(\frac{\gauss}{\pi}\right)  (-\alpha)^m(\beta)^n
  \int_{\C} \xi^m \overline{\xi}^n e^{\magn |z|^2 -\gauss|\xi|^2 +\alpha \scal{\xi,z} - \beta \overline{\scal{\xi,z}}} d\lambda(\xi).
  \end{align}
  Here $\magn = \frac{\alpha\beta}{\gauss }$ with $\mu>0$ and $\alpha,\beta\in \C$ such that $\alpha\beta >0$.
By taking for example $\gauss=1$ and $\alpha =-\beta=i$, so that $\magn =\alpha\beta/\mu=1$, the integral representation \eqref{intHermite} reduces further to the one obtained by Ismail \cite[Theorem 5.1]{IsmailTrans2016}. 

Using this integral representation and the reproducing property \eqref{intGaussab}, it is easy to obtain the following exponential generating function
  \begin{align}\label{genFctgHermite}
\sum_{m,n=0}^{\infty} \frac{u^mv^n}{m!n!} H_{m,n}^{\magn }(z;\bz) =  e^{\magn (uz +v \overline{z} -uv)} .
  \end{align}
The result \eqref{genFctgHermite} is well-known for the special case $\magn =1$ (see for example \cite{Wunsche1998,Gh13ITSF}) can also be obtained
directly by means of $(a)$ of Proposition in \cite{Gh13ITSF}, to wit
\begin{align} \label{GenFctm}
        \sum\limits_{n=0}^{+\infty}\frac{z^n}{n!} H_{n,m'}^{\nu}(w,\bw ) = \nu^{m'} (\bw - z)^{m'} e^{\magn z w}.
        \end{align}
We conclude this section by recalling the following
\begin{align}\label{genfct1hh}
\sum\limits_{n=0}^{+\infty} \frac{ H_{m,n}^{\nu}(z,\bz ) \overline{H_{m',n}^{\nu}(w,\bar w )} }{ \nu^n n! } =
 (-1)^{m'}  H_{m,m'}^{\nu}(z-w,\overline{z-w}) e^{\magn \scal{w,z}}.
\end{align}
The identity \eqref{genfct1hh} is exactly Proposition 3.6 in \cite{Gh13ITSF} (when $\nu=1$) and appears as a particular case of Theorem 3.1 in \cite{BenahmadiGElkachkouri2017}. Its proof follows making use of
$$H_{m,n}^{\nu}(z,\bz ) = (-1)^m \nu^n e^{\nu|z|^2} \pbz^m ( \bz^n   e^{-\nu |z|^2})$$
 as well as \eqref{GenFctm}. Indeed, we get
\begin{align*}
\sum\limits_{n=0}^{+\infty} \frac{ H_{m,n}^{\nu}(z,\bz ) \overline{H_{m',n}^{\nu}(w,\bar w )} }{\nu^n n! }
&= (-1)^{m+m'} e^{\nu|z|^2} \pbz^m ( \nu ^{m'} \overline{(z-w)}^{m'} e^{-\nu |z|^2+\nu \bz w})
\\ &=  (-1)^{m'} H_{m,m'}^{\nu}(z-w,\overline{z-w}) e^{\nu w \overline{z}}.
\end{align*}

\begin{remark}
By taking $m=m'$ in \eqref{genfct1hh}, we recognize the explicit expression of the reproducing kernel of the generalized Bargmann space of level $m$, defined as the $L^2$-eigenspace of the magnetic Laplacian $\Delta_\nu$ (in \eqref{MagnLap} below) associated to the eigenvalue $m$.
\end{remark}

\section{First Mehler's formula}

\begin{theorem}\label{thm:LikeMehler} For every $u,z\in \C$ such that $|u|<1$, we have
\begin{align}\label{LikeMehler}
\sum\limits_{m,n=0}^{+\infty} \frac{ u^m H_{m,n}^{\nu}(z,\bz ) \overline{H_{m,n}^{\nu}(w,\bar w )} }{\nu^{m+n} m!n! }
= \frac{e^{\magn \scal{w,z}}}{(1- u)} \exp\left(\frac{- \magn u |z-w|^2 }{1- u}\right)  .
\end{align}
\end{theorem}

\begin{proof}
By taking $m=m'$ in \eqref{genfct1hh} and using the fact that $ H_{m,m}^{\nu}(\xi,\bar\xi) =  (-1)^m m!\nu^m L^{(0)}_m(\magn |\xi|^2)$, we get
\begin{align*}
\sum\limits_{n=0}^{+\infty} \frac{ H_{m,n}^{\nu}(z,\bz ) \overline{H_{m,n}^{\nu}(w,\bar w )} }{\nu^n n! }
         =  m! \nu^m L^{(0)}_m(\magn |z-w|^2) e^{\magn \scal{w,z}}.
\end{align*}
Therefore, the identity \eqref{LikeMehler} follows making use of the well-known generating function for the Laguerre polynomials, to wit (\cite[p. 135]{Rainville71}):
$$ \sum_{n=0}^\infty t^n L^{(\alpha)}_n(x) = \frac{1}{(1-t)^{1+\alpha}} \exp\left(\frac{-xt}{1-t}\right); \quad |t|<1, x\in \R^+. $$
\end{proof}

As a particular case of Theorem \ref{thm:LikeMehler}, we have the following.

\begin{corollary}\label{cor:LikeMehler} The remarkable identity
\begin{align*}
\sum\limits_{m,n=0}^{+\infty} \frac{u^m \left|H_{m,n}^{\nu}(z,\bz )\right|^2  }{\nu^{m+n} m!n! }
= \frac{e^{\magn |z|^2 }}{(1- u)}
\end{align*}
holds true for every $u,z\in \C$ such that $|u|<1$.
\end{corollary}

An interesting application of Theorem \ref{thm:LikeMehler} is given when considering the Cauchy problem
$$ (H) \quad \left\{ \begin{array}{ll} \frac{\partial }{\partial t} u(t;z) = \Delta_{\magn} u(t;z); & (t;z)\in ]0,+\infty[\times \C, \\ u(t;z) = f(z) \in \mathcal{C}^\infty_0(\C), \end{array} \right.
 $$
 associated to the self-adjoint magnetic Laplacian
\begin{align}\label{MagnLap}
\Delta_{\magn}  =  - \dfrac{\partial^2}{\partial z\partial\bz  } + \magn z  \dfrac{\partial}{\partial z }
\end{align}
acting on $L^{2,\magn }(\C; e^{-\magn |z|^2}d\lambda)$.
In fact, we give a closed explicit expression of the Heat kernel function $K_\magn (t;z,z_0)$ for the heat solution of $(H)$. Namely, we assert

\begin{theorem}\label{thm:Heatkernel}
For every $t>0$, we have
    \begin{align}\label{Heatkernel}
K_\magn (t;z,z_0) = \left(\frac{\nu}{\pi}\right)
 \frac{e^{\magn (t+\scal{z_0,z})} }{1 - e^{\magn t}}  \exp\left( \frac{|z-z_0|^2}{e^{\magn t}-1}  \right) .
  \end{align}
\end{theorem}

\begin{proof}
Recall first that the expression of  $K(t;z,z_0)$ of a self-adjoint operator $L$ is given by
$$  K(t;z,z_0)= \sum_{j=0}^\infty e^{-\lambda_j t} e_j(z)\overline{e_j(z_0)} ,$$
where $\{e_j\}$ is a complete orthonormal system of eigenfunctions of $L$ and the $\lambda_j$ are the corresponding eigenvalues (see \cite{Davies} for example).
In our case the Hermite polynomials $H_{m,n}^\magn (z,\bz )$ constitute an orthogonal basis of $L^{2,\magn }(\C; e^{\magn |z|^2}d\lambda)$ whose the square norm is given by
\begin{align}\label{norm}
  \norm{H_{m;n}^{\magn }}^2_{L^{2}(\C; e^{-|z|^2}d\lambda)}  = \left(\frac{\pi}{\nu}\right) \magn^{m+n}m!n!
 \end{align}
 (see \cite{Ito52,IntInt06,ABEG2015}). Moreover, for varying nonnegative integer $n$, they are eigenfunctions of $\Delta_{\magn}$ associated to the eigenvalue $\magn m $, to wit $\Delta_{\magn} H_{m,n}^\magn (z,\bz ) = \magn m H_{m,n}^\magn (z,\bz )$.
 Therefore, the heat kernel function
 \begin{align*}
K(t;z,z_0)& = \frac{\nu}{\pi} \sum_{m,n=0}^{\infty} e^{-m\magn t}  \frac{ H_{m,n}^{\magn}(z;\bz) \overline{H_{m,n}^{\magn}(z_0;\bz_0)} }{\nu^{m+n} m!n!}
  \end{align*}
is given by the closed expression \eqref{Heatkernel} thanks to \eqref{LikeMehler} in Theorem \ref{thm:LikeMehler} with $w=\bz_0$ and $u=e^{-\magn t}\in \R$ such that $ e^{-\magn t} < 1$ (i.e., $t>0$).
\end{proof}

\section{Second Mehler's formula}

Consider the kernel function
\begin{align}\label{MehlerKernel}
E^{\magn,\magn'}_{u,v}(z,w) : =
 \frac{1}{1 - \magn \magn ' uv}  \exp\left(- \frac{\magn \magn ' \left[(\magn |z|^2 + \magn '|w|^2) uv - uzw - v\overline{z}\overline{w}\right] }{1 - \magn \magn ' uv} \right) .
  \end{align}
  It can be seen as an analytic extension of the classical Mehler's kernel $E_t(x,y)$ involved in \eqref{MehlerkernelHn} and introduced by Mehler himself in 1866 (\cite{Mehler1866}).
One proves that $E^{\magn,\magn'}_{u,v}(z,w)$ can be expanded in terms of the $H_{m,n}^{\magn}$. More precisely, we assert

\begin{theorem}\label{thm:Mehler}
 For every $\magn ,\magn ' \in \R$, we have the Mehler's formula
  \begin{align}
E^{\magn,\magn'}_{u,v}(z,w)  = \sum_{m,n=0}^{\infty} & \frac{u^mv^n}{m!n!} H_{m,n}^{\magn }(z;\bz) H_{m,n}^{\magn '}(w;\bw)  \label{Mehler}
  \end{align}
  valid for every $u,v \in \C$ such that $uv\in \R$ and $\magn \magn'uv < 1$.
\end{theorem}

\begin{proof}
The proof of Theorem \ref{thm:Mehler} can be handled easily starting from the right hand-side of \eqref{Mehler}. Indeed, making use of the integral representation \eqref{intHermite} combined with the generating function \eqref{genFctgHermite}, we obtain
  \begin{align*}
  \sum_{m,n=0}^{\infty}  \frac{u^mv^n}{m!n!} H_{m,n}^{\magn }(z;\bz) H_{m,n}^{\magn '}(w;\bw)
  &= \frac{\mu e^{\nu|z|^2} }{\pi} \int_{\C} e^{-\mu|\xi|^2+\alpha \bz \xi - \beta z \overline{\xi}}  \left( \sum_{m,n=0}^{\infty} \frac{(-\alpha u \xi)^m(\beta v \overline{\xi})^n}{m!n!} H_{m,n}^{\magn '}(w;\bw) \right) d\lambda(\xi)
  \\&=  \frac{\mu e^{\nu|z|^2} }{\pi} \int_{\C} e^{-\mu(1-\nu\nu'uv)|\xi|^2+\alpha( \bz -\nu' uw) \xi - \beta (z -\nu'v\bw)\overline{\xi}}  d\lambda(\xi).
  \end{align*}
 Thus, under the condition $1-\nu\nu'uv>0$, the integral formula \eqref{intGaussab} can be applied to get \eqref{Mehler}.
\end{proof}

\begin{remark}
The result gives a closed formula for the Poisson kernel (Mehler's formula) for the product of Hermite polynomials $H_{m,n}^{\magn }(z;\bz)$ associated to the same or different real arguments $\nu$ and $\nu'$.
\end{remark}

\begin{remark}\label{rem:Mehlerpc}
For the particular case $\magn =\magn '=1$, the Mehler's formula \eqref{Mehler} reads simply
\begin{align}\label{Mehlerpc1}
\sum_{m,n=0}^{\infty}  \frac{u^mv^n}{m!n!} H_{m,n}(z;\bz) H_{m,n}(w;\bw)
 =  \frac{1}{1 -  uv}  \exp\left( \frac{uzw + v\overline{z}\overline{w} -(|z|^2 + |w|^2) uv }{1 -  uv}  \right) .
  \end{align}
  This is the the Mehler's formula for $H_{m,n}(z;\bz)$ given by W\"unsche \cite{Wunsche1999} without proof and recovered by Ismail \cite[Theorem 3.3]{IsmailTrans2016} as a specific case of his Kibble-Slepian formula \cite[Theorem 1.1]{IsmailTrans2016}.
\end{remark}

By specifying $u,v,z$ and $w$ in Theorem \ref{thm:Mehler}, we obtain interesting and curious identities involving the univariate complex Hermite polynomials.

\begin{corollary}
\begin{enumerate}
\item For $z,w\in \C$ and reals $\magn , \magn '> 0$ such that $\magn \magn '< 1$, we have
 \begin{align}\label{Mehler0}
\sum_{m,n=0}^{\infty}  \frac{H_{m,n}^{\magn }(z;\bz) H_{m,n}^{\magn '}(w;\bw)}{m!n!}
=  \frac{1}{1 - \magn \magn '}  \exp\left( \frac{-\magn \magn '}{1 - \magn \magn '} (\magn |z|^2 + \magn '|w|^2- 2 \Re(zw)) \right) . 
  \end{align}
\item For every $u,v,z \in \C$ and $\magn > 0$ such that $\magn ^2uv<1$, we have
  \begin{align}\label{Mehler1}
\sum_{m,n=0}^{\infty}  \frac{u^mv^n}{m!n!} \left|H_{m,n}^{\magn }(z;\bz)\right|^2
=  \frac{1}{1 - \magn ^2 uv}  \exp\left( \frac{\magn ^2   (u + v - 2\magn uv)  }{1 - \magn ^2 uv} |z|^2 \right) . 
  \end{align}
 as well as
   \begin{align}\label{Mehler2}
\sum_{m,n=0}^{\infty} \frac{u^mv^n}{m!n!} \left( H_{m,n}^{\magn }(z;\bz)\right)^2
 = \frac{1}{1 - \magn ^2 uv}  \exp\left( \frac{\magn ^2}{1 - \magn ^2 uv} \left\{ u z^2 + v \bz^2 - 2\magn uv |z|^2 \right\} \right) . 
  \end{align}
\item For every $z\in \C$, $\magn > 0$ and real $\lambda$ such that $ \lambda\magn <1$ we have
  \begin{align}\label{Mehler3}
\sum_{m=0}^{\infty}  \frac{\lambda^m}{m!} H_{m,m}^{\magn }(z;\bz)
 = \frac{1}{1 + \lambda\magn }  \exp\left( \frac{\lambda \magn ^2 |z|^2 }{1 + \lambda\magn }    \right).
  \end{align}
  \end{enumerate}
\end{corollary}

\begin{proof}
The identity \eqref{Mehler0} follows as by taking $u=v=1$ in \eqref{Mehler} with  $\magn \magn '< 1$, while \eqref{Mehler1} and  \eqref{Mehler2}  follow respectively by setting $w=\bz$ and $w=z$ in \eqref{Mehler} with $\nu=\nu'$ and $\magn ^2uv<1$.
The last identity follows by taking $w=0$ and setting  $\lambda := -\magn 'uv$ under the assumption that $ \lambda\magn <1$, we see that \eqref{Mehler} yields
\eqref{Mehler3}. A basic fact in obtaining \eqref{Mehler3} is the following
$  H_{m,n}^{\magn'}(0;0) = (-\magn')^m m!\delta_{m,n} $.
\end{proof}

As an interesting consequence of the Mehler's formula \eqref{Mehler}, we prove the following

\begin{theorem}\label{thm:MehlerCons}
For every $\magn ,\magn ' \in \R$ and $u,v \in \C$ such that $uv\in \R$ and $\magn \magn 'uv <1$,
we have a self-reciprocity property for the Hermite polynomials, to wit
   \begin{align}\label{selfFourier}
   \int_\C    \exp&\left( \frac{-\magn'|w|^2 - \magn\magn' (uzw - v\bz\bw) }{1 - \magn \magn ' uv}  \right) H_{k,j}^{\magn '}(w;\bw) d\lambda(w)
   \\ & \qquad = \pi (\magn')^{j+k-1} (1 - \magn \magn ' uv) u^jv^k
     \exp\left(\frac{\magn ^2\magn 'uv}{1 - \magn \magn ' uv} |z|^2 \right) H_{j,k}^{\magn }(z;\bz). \nonumber
    \end{align}
\end{theorem}

\begin{proof}
  The identity \eqref{selfFourier} follows
  by multiplying the both sides of \eqref{Mehler} by $e^{\magn '|w|^2}H_{j,k}^{\magn '}(w;\bw)$ and integrating over $\C$, keeping in mind the orthogonality relation as well as the formula \eqref{norm} giving the square norm of $H_{m;n}^{\magn }$ in $L^{2}(\C; e^{-|z|^2}d\lambda)$.
\end{proof}

 \begin{remark}
 By considering the particular case of $\magn =\magn ' =1$  and letting $u,v $ tend to $i$, we obtain
 \begin{align}\label{eigenFourier1} %
    \int_\C    e^{ i \Re(zw)} e^{-\frac{|w|^2}{2}} H_{k,j}(w;\bw) d\lambda(w) = 2 \pi   i^{j+k}  e^{-\frac{|z|^2}{2}} H_{j,k}(z;\bz)
  \end{align}
The self-reciprocity property \eqref{eigenFourier1} is the analogue of the classical fact that the real Hermite functions are eigenfunctions of the Fourier transform (see for example \cite{Andrews}).
\end{remark}


\end{document}